\documentclass[11pt]{amsart}
\usepackage[latin1]{inputenc}
\usepackage[english]{babel}
\usepackage{amsmath,amssymb,amsthm,amsfonts}

\makeatletter
\@namedef{subjclassname@2010}{ \textup{2010} Mathematics Subject Classification}
\makeatother
\theoremstyle{definition}
\newtheorem{theorem}{Theorem}
\theoremstyle{definition}
\newtheorem{lemma}[theorem]{Lemma}
\theoremstyle{definition}
\newtheorem{corollary}[theorem]{Corollary}
\theoremstyle{definition}
\newtheorem{proposition}[theorem]{Proposition}
\theoremstyle{definition}
\newtheorem{fact}[theorem]{Fact}
\theoremstyle{definition}
\newtheorem{remark}[theorem]{Remark}
\theoremstyle{definition}
\newtheorem{example}[theorem]{Example}
\theoremstyle{definition}
\newtheorem{question}[theorem]{Question}
\numberwithin{equation}{section}

\numberwithin{equation}{section}
\DeclareMathOperator{\supp}{supp}

\DeclareMathOperator*{\essinf}{infess}

\newcommand{\norm}[1]{\left\lVert#1\right\rVert}
\newcommand{\abs}[1]{\left\lvert#1\right\rvert}

\begin{document}
\title{Local approach to order continuity in Ces\`aro function spaces}
\author{Tomasz Kiwerski and Jakub Tomaszewski}


\maketitle
\vspace{-9mm}

\begin{abstract}
The goal of this paper is to present a complete characterization of
points of order continuity in abstract Ces\`aro function spaces $CX$ for $X$ being a symmetric function space.
Under some additional assumptions mentioned result takes the form $(CX)_a = C(X_a)$.
We also find simple equivalent condition for this equality which in the case of $I=[0,1]$ comes to $X\neq L^\infty$.
Furthermore, we prove that $X$ is order continuous if and only if $CX$ is, under assumption that the Ces\`aro operator
is bounded on $X$.
This result is applied to particular spaces, namely: Ces\`aro-Orlicz function spaces, Ces\`aro-Lorentz function spaces and Ces\`aro-Marcinkiewicz function spaces to get criteria for OC-points.
\end{abstract}

\vspace{-9mm}

\renewcommand{\thefootnote}{\fnsymbol{footnote}}

\bigskip
\bigskip
\bigskip
\section{Introduction}

Much has been said about the Ces\`aro spaces both from isomorphic and isometric point of view, see for example the papers
\cite{AM09}, \cite{AM14} by Astashkin and Maligranda, \cite{CR16} by Curbera and Ricker, \cite{DS07} by Delgado and Soria and \cite{LM15a}, \cite{LM15b}, \cite{LM15p} by Le\'{s}nik and Maligranda and references given there.

For a Banach ideal space $X$ the abstract Ces\`aro space $CX$ is the space of all functions $f$ such that $C|f| \in X$, equipped with
the norm $\norm{f}_{CX} = \norm{C|f|}_{X}$, where $C$ denotes the continuous Ces\`aro operator
$$C : f \mapsto Cf(x) := \frac{1}{x}\int_0^x f(t)\text{d}t.$$
General considerations for this construction has been initiated in \cite{LM15a}. The abstract Ces\`aro spaces $CX$ are neither symmetric nor reflexive. Surprisingly, the descriptions of Kothe duals of $CX$ spaces is different for the case of $I = [0,1]$
and $I = [0,\infty)$, see \cite{LM15a}.

Here we will focus on considering the abstract Ces\`aro spaces $CX$ for those function spaces $X$ which are symmetric. We study the local structure of this spaces in the terms of order continuity property.

The paper is organized as follows.

In Section 2 we collect some necessary preliminaries on Banach ideal spaces, symmetric function spaces, Ces\`aro function spaces and order continuity property. Here, we specify Theorems A, B, C, D and E, Fact 1 and Lemma 2 because we will used them often in this article.

Section 3 contains the main results of this paper. Curbera and Ricker in \cite{CR16} proved that $(CX)_a = C(X_a)$ for symmetric spaces $X \neq L^\infty[0,1]$ on $I=[0,1]$. Moreover, Kiwerski and Kolwicz in \cite{KK16} have shown analogous equality in the case of the Ces\`aro-Orlicz function spaces $Ces_\Phi := CL^\Phi$, see also Remark \ref{dyskusja} for a more accurate discussion. We extend these results to the class of symmetric spaces and we get a full characterization of order continuous points in abstract Ces\`aro function spaces $CX$.

In the last Section 4, we show applications of our characterization for particular cases of symmetric spaces. Some results from this section were proved earlier directly for $X = L^p$ by Hassard and Hussein in \cite{HH73}, Shiue in \cite{Sh70}, for $X = L^\infty$ by Zaanen in \cite{Za83} and for $X = L^\Phi$ by Kiwerski and Kolwicz in \cite{KK16}.

\bigskip
\section{Notation and preliminaries}

 Denote by $m$ the
Lebesgue measure on $I$ and by $L^{0}=L^{0}(I)$ the space of all classes of
real-valued Lebesgue measurable functions defined on $I$, where $I=[0,1]$ or 
$I=[0,\infty )$. Through all the paper when we pick a subset $A \subset I$ we assume that $A$ is a Lebesgue measurable set.
For a subset $A\subset I$ we define the essential infimum $\essinf(A)$ of $A$ as follows
$$\essinf(A):=\inf\{\epsilon\geq 0:m([0,\epsilon)\cap A)=0\}.$$

A Banach space $E = (E,\Vert \cdot \Vert )$ is said to be a Banach ideal space
on $I$ (we write $E[0,1]$ or $E[0,\infty)$) if $E$ is a linear subspace of $L^{0}(I)$ and satisfies the
condition that if $g\in E$, $f\in L^{0}$ and $|f|\leq |g|$ a.e. on $I$ then $f\in E$
and $\Vert f\Vert \leq \Vert g\Vert $. Sometimes we write $\left\| \cdot\right\|_{E}$ to be sure in which
space the norm has been taken. We say that $E$ is non-trivial if $E\neq \{0\}$. 


For $f\in L^0$ we define support as $\supp(f):=\{t\in I : f(t)\neq 0\}.$
Recall that the support $\supp(X)$ of  Banach ideal space $X$ is defined as measurable subset of $I$ such that  $m(\supp(f)\setminus\supp(X)) = 0$ for every $f\in X$, and for every $U\subset\supp(X)$ of finite positive measure we have $\chi_U\in X$.

For two Banach spaces $E$ and $F$ on $I$ the symbol $E \hookrightarrow
F$ means that the embedding $E\subset F$ is continuous, i.e. there exists a
constant $A>0$ (we call it the embedding constant) such that $\left\| f\right\|_{F} \le A\left\| f\right\|_{E}$
for all $f\in E$. Recall that for two Banach ideal spaces $E$ and $F$ the embedding $E \subset F$ is always continuous. 
 Moreover, $E=F$ (resp. $E\equiv F$) means that the spaces are the same as the sets and the norms are equivalent (resp. equal). By $E\simeq F$ we denote the fact that the Banach spaces $E$ and $F$ are isomorphic.

For function $f\in L^{0}$ we define distribution functions as $d_f(\lambda) := m(\{t \in I : |f(t)| > \lambda\})$ for $\lambda \geq 0$. We say that two functions  $f,g \in L^{0}(I)$  are equimeasurable 
when they have the same distribution functions, i.e. $d_f \equiv d_y$.
By a symmetric function space (symmetric Banach function space or rearrangement invariant
Banach function space) on $I$ we mean a Banach ideal space $E = (E,\norm{\cdot}_{E})$ with
the additional property that for any two equimeasurable functions  $f,g \in L^{0}(I)$ if $f\in E$ then $g \in E$ and $\norm{f}_{E} = \norm{g}_{E}$ (we also accept the convention to write "symmetric space" within the meaning of "symmetric function space" because
in this paper we focus only on the consideration of this case). In particular, $\norm{f}_{E} = \norm{f^{*}}_{E}$, where $f^{*}(t) := \inf\{\lambda > 0 : d_f(\lambda) \leq t \}$ for $t \geq 0$. Note that, if a symmetric function space $E$ on $I$ is non-trivial then $\supp(E) = I$. For the theory of symmetric spaces the reader is referred to \cite{BS88} and \cite{KPS78}.

For a symmetric function space $E$ on $I$ its fundamental function $\varphi_E$ is
defined by the formula
$$\varphi_E(t) := \norm{\chi_{[0,t]}}_{E},$$
for $0 < t \in I$. Writing $\varphi_E(0^{+})$ or $\varphi_E(\infty)$ we
understand $\lim\limits_{t\rightarrow 0^{+}} \varphi_E(t)$ or $\lim\limits_{t\rightarrow \infty} \varphi_E(t)$, respectively.

A point $f\in X$ is said to have an order continuous norm ($f$ is an OC-point) if for any sequence $(f_n)\subset X$
with $0 \leq f_n \leq |f|$ and $f_n \rightarrow 0$ a.e. on $I$, we have $\norm{f_n}_{X} \rightarrow 0$.
By $X_{a}$ we denote the subspace of all order continuous
elements of $X$. A Banach ideal
space $X$ is called order continuous (we write $X\in (OC)$ for short) if every element of $X$ has an order continuous
norm, i.e. if $X_a = X$. It is worth to notice that in case of Banach ideal spaces
on $I$, $f\in X_{a}$ if and only if $\norm{f\chi _{A_{n}}}_{X} \rightarrow 0$ for any sequence $(A_n)$ satisfying $A_n \downarrow \emptyset$ (that is, decreasing sequence of Lebesgue measurable sets $A_{n} \subset I$ with intersection of measure zero, see \cite[Proposition 3.5, p. 15]{BS88}). The subspace $X_a$ is always closed, see \cite[Th. 3.8, p. 16]{BS88}.
Characterization of order continuity given in Theorem A (iii) and (iv) is well known.

\medskip
\noindent
\textbf{Theorem A.} \label{OC*}
\begin{enumerate}
\item (\cite[Lemma 2.6]{CKP14}) Let $E$ be symmetric space. Then $f\in E$ is a point of order continuity if and only if $f^*$ is also.
\item (\cite[Lemma 2.5]{AP02}, \cite[Lemma 2.5]{CKP14}, cf. \cite{KPS78}) Let $E$ be symmetric space. Then $f\in E_a$ if and only if
$$\norm{f^*\chi_{[0,\frac{1}{n})}}_{E}\rightarrow 0 \quad\text{ and } \norm{f^*\chi_{[n,\infty)}}_{E}\rightarrow 0
 \quad\text{ for } n\rightarrow\infty.$$
In particular if $f$ is a point of order continuity then $f^{*}(\infty) = 0$.
\item (\cite{Lo69}) A Banach ideal space $E$ is order continuous if and only if $E$ contains no
isomorphic copy of $l^\infty$.
\item (\cite[Th. 5.5, p. 27]{BS88}) A Banach ideal space $E$ is order continuous if and only if is separable.
\end{enumerate}$\hfill\ensuremath{\square}$
\medskip

Let $X$ be a Banach ideal space. The closure in $X$ of the set of simple functions is denoted by $X_b$. It is well known, that
the subspace $X_b$ is the closure in $X$ of the set of bounded functions supported in sets of finite measure, cf. \cite[Prop. 3.10. p. 17]{BS88}. Of course, $X_b$ is always non-trivial for non-trivial space $X$, moreover the subspace $X_b$ is an order ideal of $X$. We always have $X_a \subset X_b$, see \cite[Th. 3.11, p. 18]{BS88}, and
the inclusion $\{0\} \neq X_a \subset X_b$ may be proper, cf. \cite[Ex. 3, p. 30]{BS88}. The fact, that this example is based on non-symmetric construction is essential. In the symmetric case we have only two possibilities, more precisely $X_a = \{0\}$ or
$X_a = X_b$. The subspaces $X_a$ and $X_b$ coincide if and only if the characteriztic functions of the sets of finite measure
all have absolutely continuous norms, cf. \cite[Th. 3.13, p. 19]{BS88}, \cite[Prop. 2.2]{Ko05} and Theorem B. 
Next theorem describes the opposite extreme, when $X_a$ is trivial.

\medskip
\noindent
\textbf{Theorem B.} \label{Xapsutawaunki}
Let $E$ be a symmetric space. The following conditions are equivalent:
\begin{enumerate}
\item $E_a = \{0\}$,
\item $E\hookrightarrow L^\infty$,
\item $E_a \neq E_b$,
\item $\varphi_E(0^{+}) > 0$.
\end{enumerate}
In particular, if $I=[0,1]$ then condition (ii) is equivalent to the statement $E = L^\infty$.
\begin{proof}
Obviously, $E\hookrightarrow L^\infty$ is equivalent to $E = L^\infty$ if $I=[0,1]$ because the embedding
$L^\infty \hookrightarrow E$ holds for each symmetric space $E$ on $I=[0,1]$, see \cite[Corollary 6.7, p. 78]{BS88}.
Equivalence of conditions (i), (iii) and (iv) follows immediately from \cite[Th. 5.5, p. 67]{BS88}
and our discussion preceding this theorem. For the implication (ii) $\Rightarrow$ (iv) it is enough
to observe, that
\begin{align*}
\varphi_E(0^{+}) & = \lim\limits_{t\rightarrow 0^{+}} \norm{\chi_{[0,t]}}_{E}
                 \geq A\lim\limits_{t\rightarrow 0^{+}} \norm{\chi_{[0,t]}}_{L^\infty}
                 = A > 0,
\end{align*}
where $A$ is the embedding constant. Finally, suppose $\varphi_E(0^{+}) > 0$ and $E \nsubseteq L^\infty$.
This means, that there is an unbounded element $f \in E$. Then the
 $f^{*}$ is unbounded in zero and $f^{*} \in E$ by the symmetry of $E$. Therefore we can find a sequence $(t_n)\subset I$,
$t_n \rightarrow 0^{+}$ as $n \rightarrow \infty$, with $f^{*}(t_n) > n$. We have
$$ \norm{\chi_{[0,t_n]}}_{E} \leq \norm{\frac{1}{n}f^{*}\chi_{[0,t_n]}}_{E} \leq \frac{1}{n}\norm{f^{*}}_{E} \rightarrow 0,$$
as $n \rightarrow \infty$ contrary to the assumption that $\varphi_E(0^{+}) > 0$.
\end{proof}
\medskip

For $s>0$ the dilation operator $D_s$ is defined, on $L^0 (I)$, by 
\begin{equation*}
D_s f(t) := f(\frac{t}{s}) \chi_I (\frac{t}{s}) = f(\frac{t}{s}) \chi_{[0,\min\{1,s\} )} (t),
\end{equation*}
for $t\in I$, is bounded in any symmetric space $E$ on $I$ and $\left\|
D_s\right\|_{E\rightarrow E}\le \text{max}\{1,s\}$ (see \cite[p. 148]{BS88}%
). The lower and upper Boyd indices of $E$ are defined by
\begin{equation*}
p(E) := \lim \limits_{s\rightarrow \infty} \frac{\text{ln}s}{\text{ln} \left\|
D_s\right\|_{E\rightarrow E}},
\end{equation*}
\begin{equation*}
q(E) := \lim \limits_{s\rightarrow 0^{+}} \frac{\text{ln}s}{\text{ln} \left\|
D_s\right\|_{E\rightarrow E}}.
\end{equation*}
They satisfy the inequalities $1\le p(E)\le q(E) \le \infty$.
For more details see \cite{Bo69}, \cite{Bo71} and \cite{Ma85}.

Next result will be used in Section 3.

\medskip
\noindent
\textbf{Theorem C.} 
(\cite[Prop. 2.b.3, p. 132]{LT79}) Let $E$ be a symmetric function space. Then for every $1\leq p < p(E)$ and $q(E)<q\leq\infty$ we have
$$L^p\cap L^{q} \hookrightarrow E \hookrightarrow L^p+L^{q}.$$
In particular, if we take $q = \infty$ then for every $1\leq p < p(E)$ we get
$$L^p\cap L^{\infty} \hookrightarrow E\hookrightarrow L^p+L^{\infty}.$$ $\hfill\ensuremath{\square}$
\medskip

The continuous Ces\`aro operator $C : L^{0}(I)\rightarrow L^{0}(I)$ is defined
by 
\begin{equation*}
Cf(x) := \frac{1}{x}\int_{0}^{x}f(t)\text{d}t,
\end{equation*}%
for $0<x\in I$. For a Banach ideal space $X$ on $I$ we define an abstract Ces\`aro space $CX=CX(I)$ by 
\begin{equation*}
CX := \{f\in L^{0}(I):C|f|\in X\},
\end{equation*}%
with the norm $\left\Vert f\right\Vert _{CX}=\left\Vert C|f|\right\Vert _{X}$
(see \cite{CR16}, \cite{DS07}, \cite{LM15a}, \cite{LM15b}, \cite{LM15p}).
Let us note that for non-symmetric space $X$ the space $CX$ need not have a weak unit even if $X$ has it
(see \cite[Example 2]{LM15a}), so in general $\supp(CX) \subset \supp(X)$. Of course, if we assume that $C : X \rightarrow X$
then $\supp(CX) = \supp(X)$, cf. \cite[Remark 1]{LM15a}. Moreover, if $CX$ is non-trivial and $X$ is symmetric space on $I$ then
$\supp(CX) = I$, see Lemma \ref{niepustoscCX}.

Let us mention the important result about boundedness of the Ces\`aro operator.
\medskip
\newline
\textbf{Theorem D.} \cite[\text{p.} 127]{KMP07} For any symmetric space $E$
on $I$ the operator $C : E\rightarrow E$ is bounded if and only if the lower
Boyd index satisfies $p(E)>1$. $\hfill\ensuremath{\square}$
\medskip

Note that if $(X,\left\| \cdot\right\|_{X})$ is a Banach ideal space then the assumption $C
: X \rightarrow X$ is equivalent to the statement $C$ is bounded. In fact, if $C : X \rightarrow X$%
, then $X\hookrightarrow CX$. This means that there is $M>0$ with $\left\|
f\right\|_{CX} = \left\| C|f| \right\|_{X} \le M\left\| f \right\|_{X}$ for all $%
f\in X$, i.e. $C$ is bounded. On the other hand, if $C$ is bounded then for all $f \in X$ we have
$\norm{f}_{CX} = \norm{C|f|}_{X} \leq \|C\| \norm{|f|}_{X} = \|C\| \norm{f}_{X},$ which
means that $X \subset CX$ and $C : X \rightarrow X$.
However, it may happen that $%
X\not\hookrightarrow CX$ (see \cite[Proposition 2.1]{DS07}, \cite[Example 14]{KK17}). Moreover, $C :
CX \rightarrow X$ is always bounded (from the definition of $CX$) and $CX$
is so-called optimal domain of $C$ for $X$ (cf. \cite{DS07} and \cite{LM15b}%
).

The immediate consequence of Theorem D and the above discussion about
boundedness of the Ces\`aro operator is a next

\medskip
\noindent
\textbf{Theorem E.}
\label{coro: inkluzja} For any symmetric space $E$ the embedding $E \hookrightarrow CE$
holds if and only if $p(E) > 1$. In particular, if $p(E) > 1$ then the space $CE$ is non-trivial.$\hfill\ensuremath{\square}$

\medskip

It is worth to notice the following useful observation.

\begin{fact}\label{Cfgwiazdka}
(cf. \cite[Prop. 3.2, p. 52]{BS88}) For every function $f : I \rightarrow \mathbb{R}$ we have the following inequalities
$$Cf \leq \abs{Cf} \leq C\abs{f} \leq C(f^*) := f^{**}.$$
Moreover, $(Cf)^{*} \leq C(f^{*})$.
$\hfill\ensuremath{\square}$
\end{fact}

Let us note some basic fact about the non-triviality of the space $CX$ when
$X$ is symmetric function space.

\begin{lemma}\label{niepustoscCX}
Let $X$ be symmetric space on $I$. If $I=[0,1]$ then $CX\neq\{0\}$.
If $I=[0,\infty)$ then the following conditions are equivalent
\begin{enumerate}
\item $CX\neq\{0\}$,
\item $\frac{1}{x}\chi_{(t_0,\infty)}\in X$ for some $t_0 > 0$,
\item $\frac{1}{x}\chi_{(t,\infty)}\in X$ for all $t > 0$.
\end{enumerate}
\end{lemma}
\begin{proof}
We give an easy proof for the convenience of the reader. Assume that $X$ is symmetric space on $I$. 
Because $L^\infty[0,1] \hookrightarrow X[0,1]$, so we have that
$$\{0\} \neq CL^\infty[0,1] := Ces_\infty[0,1] \hookrightarrow CX[0,1],$$
see \cite[Corollary 6.7, p. 78]{BS88} and \cite[Remark]{KK16}. This means that $CX[0,1]$ is always
non-trivial. Equivalence of conditions (i) and (ii) follows from \cite[Theorem 1 (a)]{LM15a}.
Of course, (iii) implies (ii). Therefore, we will only show that (ii) implies (iii). In fact,
take $0< t < t_0$ (if $t_0 < t$ there is nothing to prove). Then
$$\frac{1}{x}\chi_{(t,\infty)} = \frac{1}{x}\chi_{(t, t_0]} + \frac{1}{x}\chi_{(t_0,\infty)} \in X,$$
because $\frac{1}{x}\chi_{(t_0,\infty)} \in X$ from the assumption and $\frac{1}{x}\chi_{(t, t_0]} \in X$
since $X$ is symmetric.
\end{proof}

\section{On the OC-points in Ces\`aro function spaces CX}

\begin{theorem} \label{glownywynik}
Let $X$ be symmetric space such that Ces\`aro operator is bounded on $X$.
Then $X\in(OC)$ if and only if $CX\in(OC)$.
\end{theorem}

The proof of the above theorem will be given at the end of this section as a consequence of  characterization of
subspace $(CX)_a$. We will begin with some observations and examples.

\begin{remark}
The implication: $X$ is order continuous then $CX$ is also, follows easily (in fact from the definition) 
without any assumption, see \cite[Lemma 1 (a)]{LM15p}. Moreover, in the case of $I=[0,1]$ Theorem \ref{glownywynik} has been proved already, see \cite[Proposition 2]{KK16}.
\end{remark}

Leaving the assumption of symmetry Theorem \ref{glownywynik} ceases to be true, as the following example shows.

\begin{example} (see \cite[Example 1]{LM15p}) If $X = L^2[0,\frac{1}{4}] \oplus L^\infty[\frac{1}{4},\frac{1}{2}] \oplus L^2[\frac{1}{2},1]$, then of course $X\notin (OC)$ but $CX = Ces_2[0,1] \in (OC)$, see Proposition \ref{OCLp}.
\end{example}

\begin{remark} \label{dyskusja}
Curbera and Ricker showed in \cite[Prop. 3.1 (c)]{CR16}, using the methods of vector measures and integral representation, that
(adapting into our notation) $C(X_a) = CX_a$ whenever $X$ is a symmetric space on $[0,1]$ and $X \neq L^\infty$. The last condition can be expressed equivalently in several ways, cf. Theorem B. Similar results are also known in the case of $I=[0,\infty)$ but only
for a certain class of the Orlicz spaces. Namely, 
suppose that $\Phi$ is an Orlicz function with $\Phi < \infty$, i.e. $b_\Phi = \infty$ and $L^\Phi$ is
the Orlicz space generated by function $\Phi$ (cf. the definitions in the
subsection \ref{Cesaro-Orlicz subsection} and references therein). From \cite[Theorem 5 (i)]{KK16} we have the following
equalities
\begin{align*}
C(L_a^\Phi)
& = \{f : C\abs{f}\in L^\varphi_{a} \}\\
& = \{f : I_\Phi(\lambda C\abs{f}) < \infty \quad \text{for all } \lambda(f) :=\lambda > 0 \}\\
& = \{f : \rho_\Phi(\lambda f) < \infty \quad \text{for all } \lambda(f) :=\lambda > 0 \} = C_\Phi,
\end{align*}
where $C_\Phi$ is the space of all order continuous elements of space $CL^{\Phi}:=Ces_\Phi$, so
$$(CL^{\Phi})_a = (Ces_\Phi)_a = C_\Phi = C(L_a^\Phi).$$
\end{remark}

\begin{example}
The above observation leads naturally to the question - is it true, that $C(X_a)=(CX)_a$ for any symmetric space $X$?
Moment of thought show us that in general that is not the case. Taking $L^\infty$ we observe that 
$$C(L^{\infty}_a) = C \{0\} = \{0\} \varsubsetneq (Ces_{\infty})_a := (CL^{\infty})_a\neq \{0\},$$ 
since elementary computation shows that, for example $\chi_{[\frac{1}{2},1)}\in (CL^{\infty})_a$.
\end{example}

However, in Theorem \ref{CXa} we will give an answer when this representation is possible and describe the space $(CX)_a$
in the other cases. We start with the following lemma.



\begin{lemma}\label{CXaneq0}
Let $X$ be symmetric space such that  $CX\neq\{0\}$ and let $A\subset I$. If $\essinf(A)>0$ and $m(A)<\infty$ then $\chi_A\in(CX)_a$.
\end{lemma}

\begin{proof}
At the beginning we will show that $\chi_{(a,b)}\in (CX)_a$, when $0<a<b<\infty$. Let  $(A_n)\subset I$ be sequence of measurable sets with $A_n\downarrow\emptyset$. Put $B_n=A_n\cap (a,b)$. Simply $m(B_n)\rightarrow 0$ when $n\rightarrow\infty$. We have
\begin{align*}
C\chi_{B_n}(x)&=\frac{1}{x}\int_0^x\chi_{B_n}(t)\text{dt}\\
&=\frac{\chi_{(a,b)}(x)}{x}\int_{(a,\min\{x,b\})\cap A_n}\text{dt}+\frac{m(B_n)}{x}\chi_{(b,\infty)}(x)\\
&\leq\frac{m(B_n)}{b}\chi_{(a,b)}(x)+ \frac{m(B_n)}{x}\chi_{(b,\infty)}(x).
\end{align*}
From Lemma \ref{niepustoscCX}, assumption $CX\neq\{0\}$ implies $\norm{\frac{1}{x}\chi_{(b,\infty)}}_X<\infty$. Therefore
\begin{align*}
\norm{\chi_{ (a,b)}\chi_{A_n}}_{CX}&=\norm{\chi_{B_n}}_{CX}\\
&=\norm{C\chi_{B_n}}_X
\leq\norm{\frac{m(B_n)}{b}\chi_{(a,b)}}_X+ \norm{\frac{m(B_n)}{x}\chi_{(b,\infty)}}_X\\
&=m(B_n)(\frac{1}{b}\norm{\chi_{(a,b)}}_X+ \norm{\frac{1}{x}\chi_{(b,\infty)}}_X)\rightarrow 0 \quad \text{as } n\rightarrow\infty. 
\end{align*}
Since $A_n$ was chosen arbitrary, we get  $\chi_{(a,b)}\in  (CX)_a$.

Let $A\subset I$ be a subset with $\essinf(A)=a>0$ and $m(A)<\infty$. Take a sequence of measurable sets $(A_n) \subset I$ with $A_n\downarrow\emptyset$. Denote $B_n=A_n\cap A$. Then, like above, $m(B_n)\rightarrow 0$ when $n\rightarrow\infty$. Note that
$$C(\chi_{B_n})\leq C(\chi_{[a,m(B_n)+a]}).$$ Since $\chi_{[a,m(B_n)+a]}\leq \chi_{[a,m(A)+a]}$ and $\chi_{[a,m(B_n)+a]}(t)\downarrow 0$ as $n \rightarrow \infty$ for a.e. $t\in I$ thus, from a previous case,  we obtain 
\begin{align*}
\norm{\chi_{A}\chi_{A_n}}_{CX}&=\norm{C(\chi_{B_n})}_{X}\\
&\leq\norm{C\chi_{[a,m(B_n)+a]}}_X\\
&=\norm{\chi_{[a,m(B_n)+a]}}_{CX}\rightarrow 0 \quad \text{as } n\rightarrow\infty. 
\end{align*}
Thus $\chi_A\in(CX)_a$.
\end{proof}

\begin{corollary}\label{CXaneq0c}
Let $X$ be symmetric space such that $CX\neq\{0\}$. If $f\in(CX)_b$ and $\essinf(A)>0$, where $A\subset I$ then $f\chi_A\in(CX)_a$.
\end{corollary}

\begin{proof}
Let $f\in (CX)_b$, $(f_n)$ be sequence of simple functions with $\lim\limits_{n\rightarrow\infty}f_n = f$ and $A\subset I$ be
like in the assumptions above. Of course, $f_n\chi_{A} \rightarrow f\chi_{A}$. Moreover, $m(\supp(f_n)\cap A)<\infty$ thus  $f_n\chi_{A} \in (CX)_a$ from Lemma \ref{CXaneq0}. Since $C : CX \rightarrow X$ is always bounded and $(CX)_a$ is closed hence $f\chi_{A} \in (CX)_a$.
\end{proof}

We will use the notation
$\Delta_0 = \Delta_0(E) :=  \{f \in E : \lim\limits_{t\rightarrow 0} Cf(t) = 0 \},$
where $E$ is a Banach ideal space.

\begin{lemma}\label{delta0subsetCXa}
Let $X$ be symmetric space. Then $(CX)_b\cap\Delta_0\subset (CX)_a.$ 
\end{lemma}

\begin{proof}
If $CX=\{0\}$ there is nothing to prove.
Assume that $CX\neq\{0\}$ and take $0\leq f\in (CX)_b\cap\Delta_0$. Let $(A_n) \subset I$ be arbitrary sequence of measurable subsets with $A_n\downarrow\emptyset$. Take $\epsilon > 0$.
Since $Cf(0^+) = 0$ there exist $\delta>0$ such that $\sup_{t\in[0,\delta)}Cf(t)<\epsilon$. 
From Corollary \ref{CXaneq0c} we get $f\chi_{[\delta,\infty)}\in(CX)_a$. Therefore there exist $n_{\delta}\in\mathbb{N}$ with
$$\norm{f\chi_{[\delta,\infty)}\chi_{A_n}}_{CX}\leq\epsilon,$$
for every $n\geq n_{\delta}$.
Without loss of generality suppose that $\delta \leq 1$. Then 
$$C(f\chi_{A_n\cap[0,\delta)}) \leq \sup\{Cf(s) : s\in[0,\delta)\}(\chi_{[0,1)}+\frac{1}{t}\chi_{[1,\infty)})
\leq \epsilon(\chi_{[0,1)}+\frac{1}{t}\chi_{[1,\infty)}),$$
for $n\in\mathbb{N}$.
Moreover, $\norm{\frac{1}{t}\chi_{[1,\infty)}}_{X}<\infty$ in view of our assumption $CX\neq\{0\}$ and Lemma \ref{niepustoscCX}. We have
\begin{align*}
\norm{f\chi_{A_n}}_{CX}&\leq\norm{f\chi_{[\delta,\infty)\chi_{A_n}}}_{CX}+\norm{f\chi_{[0,\delta)\chi_{A_n}}}_{CX}\\
&\leq\epsilon+\epsilon\norm{\chi_{[0,1)}+\frac{1}{t}\chi_{[1,\infty)}}_{X}\\
&\leq\epsilon(1+\varphi_X(1)+\norm{\frac{1}{t}\chi_{[1,\infty)}}_{X}),
\end{align*}
for $n\geq n_{\delta}$.
Since $\epsilon > 0$ was arbitrary, we conclude that
$$\norm{f\chi_{A_n}}_{CX}\rightarrow 0 \quad \text{as } n\rightarrow\infty,$$
which ends the proof.
\end{proof}

\begin{lemma}\label{CXasubsetCXa}
Let $X$ be Banach ideal function space, then $$C(X_a)\subset (CX)_a.$$
\end{lemma}
\begin{proof} 
Suppose $0\leq f\in C(X_a)$. Let $(A_n)$ be arbitrary sequence of measurable subsets of $I$ with $A_n\downarrow\emptyset$. First, note that $f(t)\chi_{A_n}(t)\downarrow 0$  as $n \rightarrow \infty$ for a.e. $t\in I$ thus, from Lebesgue dominated convergence theorem, we obtain  that  $C(f\chi_{A_n})(t)\downarrow 0$ as $n \rightarrow \infty$ for a.e. $t\in I$. Since $Cf\in X_a$ and $C(f\chi_{A_n})\leq Cf$ for every $n\in\mathbb{N}$ we have
$$\norm{f\chi_{A_n}}_{CX} = \norm{C(f\chi_{A_n})}_{X} \rightarrow 0 \quad \text{as } n \rightarrow \infty,$$ i.e. $f\in (CX)_a$.
\end{proof}

\begin{lemma}\label{CXbsubsetCXa}
Let $X$ be symmetric function space such that Ces\`aro operator is bounded on $X$ and $X_a=X_b$. Then $(CX)_b\subset C(X_a)$. 
\end{lemma}
\begin{proof}
We divide the proof into three steps. First, we will show that  $\chi_{[0,a]}\in  C(X_a)$ for every $0<a<\infty$. Fix $a>0$ and note that 
$$C(\chi_{[0,a]})(t)=\chi_{[0,a]}(t)+\frac{a}{t}\chi_{[a,\infty]}(t),$$
for $0 < t \in I$.
Since $X_a=X_b$ thus $\chi_{[0,a]}\in X_a$ and we only need to show that $f(t):=\frac{1}{t}\chi_{[a,\infty]}(t)\in X_a$.
 Let  $(A_n)$ be sequence of measurable subsets of $I$ with $A_n\downarrow\emptyset$. Without loss of generality we can assume that $A_n\subset\supp(f)=[a,\infty)$ for every $n\in\mathbb{N}$. Let $\delta>0$ be arbitrary.
From Lemma \ref{CXaneq0} we have $\chi_{[a,a+\delta]}\in X_a$ therefore there exist
$n_{\delta}\in\mathbb{N}$ such that 
$\norm{\frac{1}{a}\chi_{[a,a+\delta]}\chi_{A_n}}_X\leq\frac{1}{\delta}$ for all $n\geq n_{\delta}$. 
From Theorem C and our assumption that Ces\`aro operator is bounded there exist $p>1$ with $L^p\cap L^{\infty}\hookrightarrow X$. Since $L^p\in (OC)$, we can find  $n_{\delta}'\in\mathbb{N}$ satisfying
$$\norm{\frac{1}{t}\chi_{[\delta,\infty]}\chi_{A_n}}_{L^p}\leq\frac{1}{\delta},$$
for every $n\geq n_{\delta}'$. Let  $n\geq \max\{n_{\delta},n_{\delta}'\}$ we have
\begin{align*}
\norm{f\chi_{A_n}}_X&\leq \norm{\frac{1}{a}\chi_{[a,a+\delta]}\chi_{A_n}}_X+\norm{\frac{1}{t}\chi_{[\delta,\infty]}\chi_{A_n}}_X\\
&\leq\frac{1}{\delta}+D\norm{\frac{1}{t}\chi_{[\delta,\infty]}\chi_{A_n}}_{L^{\infty}\cap L^p}\\
&=\frac{1}{\delta}+D\max\{\frac{1}{\delta},\norm{\frac{1}{t}\chi_{[\delta,\infty]}\chi_{A_n}}_{L^p}\}\\
&\leq \frac{1}{\delta} + D\max\{\frac{1}{\delta},\frac{(p-1)^{-\frac{1}{p}}}{\delta^{1 - \frac{1}{p}}} \},
\end{align*}
where $D > 0$ is inclusion constant. Since $\delta > 0$ was arbitrary, we obtain
$$\norm{f\chi_{A_n}}_X\rightarrow 0 \quad \text{as } n \rightarrow \infty.$$
Consequently,  $f\in X_a$ and hence $\chi_{[0,a]}\in  C(X_a)$.

Now let $A\subset I$ be arbitrary set of finite measure. From Fact \ref{Cfgwiazdka} we have 
$$C(\chi_{A}) \leq C(\chi_{[0, m(A)]}).$$
From a previous step of proof we also know that $C(\chi_{[0, m(A)]}) \in X_a$ thus from ideal property
of $X_a$ we get $C(\chi_{A}) \in X_a$, i.e. $\chi_{A} \in C(X_a).$

Finally, let $f\in (CX)_b$ and $(f_n)$ be a sequence of simple functions with $\lim\limits_{n\rightarrow\infty}f_n=f$ pointwise.
We already know that  $Cf_n\in X_a$ for every $n\in\mathbb{N}$. From continuity of Ces\`aro operator $C : CX \rightarrow X$ we have that $Cf_n\rightarrow Cf$. Therefore $Cf\in X_a$ because $X_a$ is closed and the proof is complete.
\end{proof}

\begin{remark} \label{Rownowaznywarunek}
We can prove more, namely: if $X$ is a symmetric function space with $X_a = X_b$ then
$(CX)_b\subset C(X_a)$ if and only if $\frac{1}{t}\chi_{(\lambda_0, \infty)}(t) \in X_a$ for some $\lambda_0 > 0$. In above proof we in fact proved the sufficiency.
The necessity is even simpler,  if $\frac{1}{t}\chi_{(1, \infty)}(t) \notin X_a$ then $C(\chi_{(0,1)})(t) \notin X_a$,
i.e. $\chi_{(0,1)}(t) \notin C(X_a)$, but of course $\chi_{(0,1)}(t) \in (CX)_b$.
It follows from the proof of remaining lemma that if $C : X \rightarrow X$ and $X_a = X_b$ then $\frac{1}{t}\chi_{[1,\infty)}\in X_a$.
Is also worth noting that in the case of $I=[0,1]$ the assumption $X \neq L^\infty$ is equivalent to $\frac{1}{t}\chi_{(\lambda, \infty)}(t) \in X_a$ for all $0 < \lambda < 1$, see Theorem B.
\end{remark}
\begin{lemma}\label{CXasubsetdelta0}
Let $X$ be symmetric space with $X_a=\{0\}$. Then $(CX)_a\subset \Delta_0.$ 
\end{lemma}
\begin{proof}
If $CX=\{0\}$ then inclusion is trivial, so we can assume that $CX\neq\{0\}$. Take $0\leq f\notin\Delta_0$. Then
$$\limsup\limits_{t\rightarrow 0^+}Cf(t)=\delta>0.$$
We can find a sequence $(a_n)\subset (0,1)$ such that $a_n\rightarrow 0$ and $Cf(a_n)>\frac{\delta}{2}$ for every $n\in\mathbb{N}$. Since $Cf$ is continuous function, therefore there exist open neighborhood $U_n$ of $a_n$ with $\inf\limits_{t\in U_n}\{Cf(t)\} \geq  \frac{\delta}{4}$ for $n\in\mathbb{N}$. We can chose a subsequence $n_k\in\mathbb{N}$ such that $U_{n_k}\cap[0,\frac{1}{k})\neq \emptyset$ for  $k\in\mathbb{N}$. Now we have
\begin{align*}
\norm{f\chi_{[0,\frac{1}{k})}}_{CX}&= \norm{C(f\chi_{[0,\frac{1}{k})})}_X\\
&=\norm{Cf(t)\chi_{[0,\frac{1}{k})}+\frac{Cf(\frac{1}{k})}{t}\chi_{[\frac{1}{k},\infty)}}_X\\
&\geq\norm{Cf(t)\chi_{[0,\frac{1}{k})}}_X\\
&\geq\norm{\frac{\delta}{4}\chi_{U_{n_k}\cap[0,\frac{1}{k})}}_X
\geq\frac{\delta}{4}\varphi_X(0^+),
\end{align*}
which shows that $f\notin (CX)_a$.
\end{proof}

\begin{lemma}\label{Cxa=CXb2}
Let $X$ be symmetric space with $X_a\neq\{0\}$. Then $$(CX)_a=(CX)_b.$$
\end{lemma}
\begin{proof}
Take $\delta > 0$. Let $(f_n) \subset (CX)_b$ be a sequence chosen so that $f_n \rightarrow f$ almost everywhere. Using Corollary \ref{CXaneq0c} we
conclude that $f\chi_{[\delta, \infty)} \in (CX)_a$. Since $C(X_a) \subset (CX)_a$ 
and
$X_a\neq\{0\}$ we have
$$\chi_{[0,\delta)} \in C(X_a) \subset (CX)_a.$$
Observe that
$$f_n\chi_{[0,\delta)} \leq f_n^{*}(0)\chi_{[0,\delta)},$$
thus $f_n\chi_{[0,\delta)} \in (CX)_a$ using the ideal property of $(CX)_a$. From continuity of $C : CX \rightarrow X$ we have 
$f_n\chi_{[0,\delta)} \rightarrow f\chi_{[0,\delta)}$ in $CX$. This means that $f\chi_{[0,\delta)} \in (CX)_a$ because $(CX)_a$
is closed.
\end{proof}

\begin{theorem}\label{CXa}
Let $X$ be symmetric function space. Then one of following holds:
\begin{enumerate}
\item $(CX)_a=(CX)_b$ if $X_a\neq\{0\}$,
\item $(CX)_a=(CX)_b\cap\Delta_0$ if $X_a=\{0\}$.
\end{enumerate}
In particular, if Ces\`aro operator is bounded on $X$ and $X_a\neq\{0\}$ then $$(CX)_a=C(X_a).$$
\end{theorem}

\begin{proof}
Suppose $X_a\neq\{0\}$. Then $X_a=X_b$ from  Theorem B. This is the case of Lemma \ref{Cxa=CXb2} so we have $(CX)_a=(CX)_b$. If we additionally assume that the Ces\`aro operator $C$ is bounded on $X$, using Lemma \ref{CXasubsetCXa} and Lemma \ref{CXbsubsetCXa} we have
$$C(X_a)\subset (CX)_a = (CX)_b\subset C(X_a).$$
Therefore $(CX)_a=C(X_a)$.

Assume now that $X_a=\{0\}.$  From Lemma \ref{CXasubsetdelta0} we have $(CX)_a\subset \Delta_0$. Since we always have that
$(CX)_a\subset(CX)_b $ hence $(CX)_a\subset (CX)_b\cap\Delta_0$. Combining this with Lemma \ref{delta0subsetCXa} we obtain
$$(CX)_b\cap\Delta_0\subset (CX)_a\subset (CX)_b\cap\Delta_0,$$
thus  $(CX)_a=(CX)_b\cap\Delta_0$.
\end{proof}

\begin{remark} \label{CXawarunekrownowazny1/x}
The previous theorem can be formulated in a more concise form if we assume that the Ces\`aro operator is bounded on
symmetric space $X$. In this case 
$$(CX)_a=C(X_a)+(CX)_b\cap\Delta_0.$$
Moreover, it follows from Remark \ref{Rownowaznywarunek} that if $X$ is symmetric space and $X_a$ is non-trivial then $(CX)_a=C(X_a)$ if and only if 
$\frac{1}{t}\chi_{(\lambda_0, \infty)}(t) \in X_a$ for some $\lambda_0 > 0$. Additionally, as consequence of Remark \ref{Rownowaznywarunek}, we have $(CX)_a=C(X_a)$  if and only if $X \neq L^\infty$ in the case of $I=[0,1]$.
\end{remark}

\begin{lemma} \label{fgwiazdka}
Let $X$ be a symmetric space such that the Ces\`aro operator is bounded on $X$. If $f \in  X\setminus X_a$ then $f^* \in  CX \setminus (CX)_a$.
\end{lemma}
\begin{proof}
Take an element $f \in  X\setminus X_a$. From a symmetry of $X$ we get $f^{*} \in X$ and $C(f^*)\in X$
since the operator $C$ is bounded on $X$. Now we have to consider two cases:
\begin{enumerate}
\item $X_a\neq\{0\}$. Then, from Theorem \ref{CXa}, $(CX)_a=C(X_a)$. From Theorem A (i) $f^*\notin  X_a$.  From ideal property of $X_a$ and since $C(f^*)\geq f^*$, we have that  $C(f^*)\notin  X_a$, i.e. $$f^*\notin  C(X_a)=(CX)_a.$$
\item $X_a=\{0\}$. Observe that $0\neq f$, since $0\in  X_a$ thus $$\limsup\limits_{t\rightarrow 0^+}C(f^*)(t)\geq\limsup\limits_{t\rightarrow 0^+}f^*(t)>0.$$ From Theorem \ref{CXa} we obtain $f^*\notin \Delta_0\cap (CX)_b=(CX)_a$. 
\end{enumerate}
\end{proof}

It is time to give proof of the Theorem \ref{glownywynik} announced at the beginning of this section.
\bigskip

\noindent
\textit{Proof of Theorem \ref{glownywynik}.}
\textit{Necessity.} Let $X$ be a symmetric space such that Ces\`aro operator is bounded on $X$ and suppose $X \in (OC)$.
Then $X = X_a$ and from Theorem \ref{CXa} we get
$$(CX)_a = C(X_a) = CX,$$
which means that $CX \in (OC)$.
\newline
\textit{Sufficiency.} If $X\notin(OC)$ then there exist an element $f \notin X_a$. From Lemma \ref{fgwiazdka} $f^* \notin  (CX)_a$  and consequently $CX \notin (OC)$ which completes the proof. $\hfill\ensuremath{\square}$

\section{Applications}

Although each of the following spaces belongs to the class of symmetric spaces, in this special
cases, the criteria for OC-points become more specified.

\subsection{The Ces\`aro function spaces $Ces_p$}

\begin{remark} \label{Eapuste}
We have the following characterization of the closure of the set of simple functions in the space $L^\infty$
$$(L^\infty[0,1])_b = L^\infty[0,1],$$
and
$$(L^\infty[0,\infty))_b = \{f\in L^\infty[0,\infty) : \lim\limits_{t\rightarrow \infty} f(t) = 0\}.$$
Of course, $(L^\infty)_a = \{0\}$ and $p(L^\infty) > 1$. Let us define a set $\Delta_\infty = \Delta_\infty(E) := \{f\in E : \lim\limits_{t\rightarrow \infty} Cf(t) = 0\}$ for $E$ being a Banach ideal space on $[0,\infty)$. With this notation, we get the following reformulation of Theorem \ref{CXa} (ii):
$$
(Ces_\infty)_a = \left\{ \begin{array}{ll}
\Delta_0 = \{f\in Ces_\infty : \lim\limits_{t\rightarrow 0^{+}} Cf(t) = 0 \} & \textrm{if $I = [0,1]$}\\
\Delta_0 \cap \Delta_\infty = \{f\in Ces_\infty : \lim\limits_{t\rightarrow 0^{+},\infty} Cf(t) = 0 \} & \textrm{if $I = [0,\infty).$}
\end{array} \right.
$$
The space $Ces_\infty[0,1]$ is known as Korenblyum-Krein-Levin space $K$ and it is also known that
$$ K_a = \{f \in K : \lim\limits_{h\rightarrow 0^+} \frac{1}{h}\int_{0}^{h}|f|\text{d}m = 0 \} = \Delta_0(K) ,$$
see \cite[pp. 469-471]{Za83}. Thus, the above characterization is a generalization of this classical result.
\end{remark}

It is worth to mention that $Ces_1[0,1] \equiv L^1(\ln(1/t))$. From Lebesgue dominated convergence theorem we obtain that
$L^1(\ln(1/t)$ is order continuous.  Moreover, $Ces_1[0,\infty) = \{0\}$, see \cite[Th. 3.1 (a)]{AM14}. Since the spaces $L^{p}$ are order continuous for $1 \leq p < \infty$ and $L^\infty_a = \{0\}$, so as a consequence of Theorem \ref{CXa} and Remark \ref{Eapuste} we obtain well known result (cf. \cite[Th. 3.1 (b)]{AM14}).

\begin{proposition} \label{OCLp} We have the following characterization
$$
(Ces_p)_a = \left\{ \begin{array}{ll}
Ces_p & \textrm{when $1 \leq p < \infty$}\\
\{f\in Ces_\infty : \lim\limits_{t\rightarrow 0^{+},\infty} Cf(t) = 0 \} & \textrm{when $p = \infty .$}
\end{array} \right.
$$
In particular, the Ces\`aro function space $Ces_p$ is order continuous if and only if $1 \leq p < \infty$.
\end{proposition}

\subsection{The Ces\`aro-Orlicz function spaces $Ces_\Phi$} \label{Cesaro-Orlicz subsection}

A function $\Phi :[0,\infty )\rightarrow \lbrack 0,\infty ]$ is called an
Orlicz function if:

\begin{enumerate}
\item $\Phi$ is convex,
\item $\Phi (0)=0$, $\Phi (\infty)=\infty$,
\item $\Phi $ is neither identically equal to zero nor infinity on $(0,\infty )$.
\end{enumerate}
For more information about Orlicz functions see \cite{Ch96} and \cite{KR61}.

Let us denote by
$$a_{\Phi } = \sup \{u\geq 0:\Phi (u)=0\},$$
and
$$b_{\Phi } = \sup \{u>0:\Phi (u)<\infty \}.$$
We have, $a_{\Phi }<\infty $ and $b_{\Phi }>0$, since an Orlicz function is
neither identically equal to zero nor infinity on $[0,\infty )$. The
function $\varphi $ is continuous and non-decreasing on $[0,b_{\Phi })$
and is strictly increasing on $[a_{\Phi },b_{\Phi })$. 

We say an Orlicz function $\Phi $ satisfies the condition $\Delta _{2}$
for large arguments ($\Phi \in \Delta _{2}(\infty )$ for short) if there
exists $K>0$ and $u_{0}>0$ such that $\Phi (u_{0})<\infty $ and 
\begin{equation*}
\Phi (2u)\leq K\Phi (u)
\end{equation*}%
for all $u\in \lbrack u_{0},\infty )$. Similarly, we can define the
condition $\Delta _{2}$ for small, with $\Phi (u_{0})>0$ $(\Phi \in
\Delta _{2}(0))$ or for all arguments $(\Phi \in \Delta _{2}(\mathbb{R_{+}%
}))$. These conditions play a crucial role in the theory of Orlicz spaces,
see \cite{Ch96}, \cite{KR61}, \cite{Ma89} and \cite{Mu83}.

The Orlicz function space $L^{\Phi }=L^{\Phi }(I)$ generated by an
Orlicz function $\Phi $ is defined by 
\begin{equation*}
L^{\Phi } := \{f\in L^{0}(I):I_{\Phi }(f/\lambda )<\infty \ \text{for some%
}\ \lambda =\lambda (f)>0\},
\end{equation*}%
where $I_{\Phi }(f)=\int_{I}\Phi (|f(t)|)\text{d}t$ is a convex modular (for
the theory of Orlicz spaces and modular spaces see \cite{Ma89} and \cite%
{Mu83}). The space $L^{\Phi }$ is a Banach ideal function space with the
Luxemburg-Nakano norm 
\begin{equation*}
\left\Vert f\right\Vert _{\Phi }=\inf \{\lambda >0:I_{\Phi }(f/\lambda
)\leq 1\}.
\end{equation*}

Let us recall that
$$L^\Phi_b = \{f\in Ces_\Phi : \forall_{\lambda > 0} \exists_{s_\lambda \in S_0}\ \text{such that}\ I_\Phi(\lambda(f-s)) < \infty \},$$
where $S_0$ is the ideal of bounded functions with support of finite measure. 
It is also easy to see that
$$L^\Phi_b = \{f\in L^\Phi : I_\Phi(\lambda f) < \infty\ \text{for all}\  \lambda > 0\} = L^\Phi_a,$$
when $b_\Phi = \infty$, cf. \cite{GH99}, \cite[Th. 3.3 (a), p. 17]{Ma89} and Theorem B.

The Ces\`aro-Orlicz function space $Ces_{\Phi }=Ces_{\Phi }(I)$ is
defined by $Ces_{\Phi }(I) := CL^{\Phi }(I)$. Consequently, the norm in
the space $Ces_{\Phi }$ is given by the formula 
\begin{equation*}
\left\Vert f\right\Vert _{Ces_\Phi}=\inf \{\lambda >0:\rho _{\Phi
}(f/\lambda )\leq 1\},
\end{equation*}%
where $\rho _{\Phi }(f)=I_{\Phi }(C|f|)$ is a convex modular.
Note that $Ces_{\Phi }[0,1]\neq \{0\}$ for any Orlicz function $\Phi $,
see Lemma \ref{niepustoscCX}, \cite[Remark 4]{KK16}. Moreover, for the equivalent conditions
to the non-triviality of $Ces_\Phi[0,\infty)$ see \cite[Proposition 3]{KK16}.

\begin{lemma} \label{opiscesfib}
Let $\Phi$ be an Orlicz function. Then
$$(Ces_\Phi)_b = \{f\in Ces_\Phi : \forall_{\lambda > 0} \exists_{s_\lambda \in S_0}\ \rho_\Phi(\frac{f-s_\lambda}{\lambda}) < \infty \}. $$
\end{lemma}
\begin{proof}
The idea of this proof comes from \cite[Th. 3.1]{GH99}. Without loss of generality, we can consider only the case when the
Ces\`aro-Orlicz function space is non-trivial.
Like in mentioned article we use the notation
$$H(S_0) =  \{f\in Ces_\Phi : \forall_{\lambda > 0} \exists_{s_\lambda \in S_0}\ \rho_\Phi(\frac{f-s_\lambda}{\lambda}) < \infty \}.$$
We claim that:
\begin{enumerate}
\item $H(S_0)$ is closed,
\item $H(S_0) \subset \text{cl}{S_0}$, where by $\text{cl}{S_0}$ we mean the closure of $S_0$ in $Ces_\Phi$.
\end{enumerate}

Take $f \in \text{cl}H(S_0)$ and fix $\lambda > 0$. There exist a sequence $(f_n) \subset H(S_0)$ with $f_n \rightarrow f$.
From \cite[Th. 1.3, p. 8]{Ma89} there exist $n_0 \in \mathbb{N}$ such that 
$$\rho_\Phi(\frac{f-f_{n_0}}{\lambda/2}) < \infty.$$
Let $s\in S_0$ be such that 
 $$\rho_\Phi(\frac{f_{n_0}-s}{\lambda/2}) < \infty.$$
Since modular $\rho_\Phi$ is convex, we have
\begin{align*}
 \rho_\Phi(\frac{f-s}{\lambda})&=\rho_\Phi(2\frac{(f-f_{n_0}+f_{n_0}-s)}{2\lambda})\\
 &\leq\frac{1}{2}(\rho_\Phi(\frac{f-f_{n_0}}{\lambda/2})+\rho_\Phi(\frac{f_{n_0}-s}{\lambda/2}))<\infty,
\end{align*}
and (i) follows.

Pick $f \in H(S_0)_+$ and $\lambda > 0$. From the definition of $H(S_0)$ there exist $0\leq s_\lambda \leq f$ such that
$$\rho_\Phi(\frac{f-s_\lambda}{\lambda}) < \infty.$$
Take a sequence $(y_n) \subset S_0$ with $y_n \uparrow f - s_\lambda$ a.e. on $I$. Since $0 \downarrow f - s_\lambda - y_n \leq f - s_\lambda$ a.e. on $I$ thus using Lebesgue dominated convergence theorem we get $C(f - s_\lambda - y_n) \downarrow 0$ a.e. on $I$.
Because $L^1\in (OC)$ and
$$\rho_\Phi(\frac{f-s_\lambda-y_n}{\lambda}) \leq \rho_\Phi(\frac{f-s_\lambda}{\lambda}) < \infty,$$
for all $n \in \mathbb{N}$, so
$$\rho_\Phi(\frac{f-s_\lambda-y_n}{\lambda}) \rightarrow 0.$$
Therefore there  exist $n_0 \in \mathbb{N}$ with
$$\rho_\Phi(\frac{f-s_\lambda-y_{n_0}}{\lambda}) < \lambda,$$
i.e. $\text{dist}(f, S_0) \leq \lambda$, see \cite{GH99}. Since $\lambda > 0$ was arbitrary so $f\in \text{cl}{S_0}$. This completes the proof of part (ii).

It is clear that $\text{cl}{S_0} = (Ces_\Phi)_b$ and $S_0 \subset H(S_0)$. Combining this observation with (i) and (ii) we conclude that
$$H(S_0) \subset \text{cl}S_0 = (Ces_\Phi)_b,$$
and
$$(Ces_\Phi)_b = \text{cl}{S_0} \subset \text{cl}H(S_0) = H(S_0),$$
i.e. $(Ces_\Phi)_b = H(S_0)$.
\end{proof}

Next proposition is a generalisation of Theorem 5, Theorem 7 and Corollary 8 from \cite{KK16}.

\begin{proposition} Let $\Phi$ be an Orlicz function. Then
$$
(Ces_\Phi)_a = \left\{ \begin{array}{ll}
\{f\in Ces_\Phi :\rho_\Phi(\lambda f) < \infty\ \text{for all}\  \lambda > 0\} & \textrm{if $b_\Phi = \infty$}\\

\{f\in Ces_\Phi : \forall_{\lambda > 0} \exists_{s_\lambda \in S_0}\ \rho_\Phi(\lambda(f-s)) < \infty\ \text{and}\ \lim\limits_{t\rightarrow 0^{+}} Cf(t) = 0 \} & \textrm{if $a_\Phi = 0, b_\Phi < \infty$}\\
\{f\in Ces_\Phi : \lim\limits_{t\rightarrow 0^{+},\infty} Cf(t) = 0 \} & \textrm{if $a_\Phi > 0, b_\Phi < \infty$}
\end{array} \right.
$$
Moreover, if $p(L^\Phi)>1$ then the Ces\`aro-Orlicz function space $Ces_\Phi$ is order continuous if and only if $\varphi \in \Delta_2$ (i.e. if $p(L^\Phi)>1$ then $Ces_\Phi \in (OC)$ if and only if $L^\Phi \in (OC)$).
\end{proposition}
\begin{proof}
Suppose $Ces_\Phi$ is non-trivial and
\begin{enumerate}
\item $b_\Phi = \infty$. Then there exist $\lambda_0 > 0$ with $I_\Phi(\frac{\eta_0}{x}\chi_{(\lambda_0,\infty)}) < \infty$ for
some $\eta_0 > 0$. We will show that this inequality holds for all $\eta > 0$. There exist $\gamma>0$ such that  $I_\Phi(\frac{1}{x}\chi_{(\gamma,\infty)}) < \infty$. Indeed, taking $\gamma=\frac{\lambda_0}{\eta_0}$ we have
\begin{align*}
I_\Phi(\frac{1}{x}\chi_{(\gamma,\infty)})&=\int_{\frac{\lambda_0}{\eta_0}}^\infty \Phi (\frac{1}{x})\text{d} x
=\eta_0^{-1}\int_{\lambda_0}^\infty \Phi (\frac{\eta_0}{u})\text{d} u\\
&=\eta_0^{-1}\int_0^\infty \Phi (\frac{\eta_0}{u}\chi_{(\lambda_0,\infty)}(u))\text{d} u
=\eta_0^{-1}I_\Phi(\frac{\eta_0}{x}\chi_{(\lambda_0,\infty)})<\infty.
\end{align*}
Take $\eta > 0$. Then
\begin{align*}
I_\Phi(\frac{\eta}{x}\chi_{(\lambda_0,\infty)})&=\eta\int_0^\infty \Phi (\frac{1}{u}\chi_{(\frac{\lambda_0}{\eta},\infty)}(u))\text{d} u\\
&\leq  \eta\int_0^\gamma \Phi (\frac{\eta}{\lambda_0})\text{d} u+\eta\int_\gamma^\infty \Phi (\frac{1}{u}\chi_{(\frac{\lambda_0}{\eta},\infty)}(x))\text{d} u\\
&\leq\eta\int_0^\gamma \Phi (\frac{\eta}{\lambda_0})\text{d} u+\eta\int_\gamma^\infty \Phi (\frac{1}{u}\chi_{(\gamma,\infty)}(u))\text{d} u<\infty,
\end{align*}
which proves the claim. Therefore,
$\frac{1}{x}\chi_{(\lambda_0,\infty)}(x) \in L^\Phi_a$. Using Remark \ref{Rownowaznywarunek} we get
$$(Ces_\Phi)_a = C(L^\Phi_a) = \{f\in Ces_\Phi :\rho_\Phi(\lambda f) < \infty\ \text{for all}\  \lambda > 0\}.$$
\item $a_\Phi = 0$ and $b_\Phi < \infty$. In this case $L^\Phi_a = \{0\}$, thus from Theorem \ref{CXa} (ii) and Lemma \ref{opiscesfib}
we get
\begin{align*}
(Ces_\Phi)_a & = (Ces_\Phi)_b \cap \Delta_0 \\ & = \{f\in Ces_\Phi : \forall_{\lambda > 0} \exists_{s_\lambda \in S_0}\ \rho_\Phi(\lambda(f-s)) < \infty\ \text{and}\ \lim\limits_{t\rightarrow 0^{+}} Cf(t) = 0 \}.
\end{align*}
\item $a_\Phi > 0$ and $b_\Phi < \infty$. Then $L^\Phi = L^\infty$, cf. \cite[Ex. 1, \text{p.} 98]{Ma89} thus $Ces_\Phi = Ces_\infty$  and it is enough to apply Remark \ref{Eapuste}.
\end{enumerate}
Moreover, the Orlicz space $L^\Phi$ is order continuous if and only if $\Phi \in \Delta_2$, see \cite[Remark 1, p. 22]{Ma89}.
Therefore, if $p(L^\Phi)>1$ then the Ces\`aro-Orlicz function space $Ces_\Phi$ is order continuous if and only if $\varphi \in \Delta_2$ from Theorem \ref{glownywynik}.
\end{proof}

\begin{remark}
Note that $p(L^\Phi) = \alpha_\Phi$ and $q(L^\Phi) = \beta_\Phi $, where $\alpha_\Phi$ and $\beta_\Phi$ are the so-called lower
and upper Orlicz-Matuszewska indices, see \cite[Prop. 2.b.5 and Remark 2, p. 140]{LT79}) and \cite{Ma89}, \cite{Mu83} for more
information.
\end{remark}

\subsection{The Ces\`aro-Lorentz function spaces $C\Lambda_\varphi$}

The fundamental function $\varphi_E$ of every symmetric function space $E$ on $I$ is quasi-concave on $I$, that is,
$\varphi_E(t) = 0$ if and only if $t=0$, $\varphi_E(t)$ is increasing on $I$ and $\varphi_E(t)/t$ is non-increasing
for $t \in (0,m(I))$. Moreover, for any quasi-concave function $\varphi$ there is a symmetric function space whose
fundamental function is $\varphi$. The smallest symmetric function space with fundamental function $\varphi$ is
called the Lorentz function space $\Lambda_\varphi = \Lambda_\varphi(I)$ and is defined as
$$\Lambda_\varphi(I) := \{ f\in L^{0}(I) : \int_0^{m(I)} f^{*} \text{d}\varphi < \infty \},$$
with a norm given as $\norm{f}_{\Lambda_\varphi(I)} = \int_0^{m(I)} f^{*} \text{d}\varphi$.
The fundamental function of a symmetric function space $E$ is not necessary concave but $E$ can be equivalently renormed
in such a way that the resulting fundamental function is concave. In this case the Riemann-Stieltjes integral in the
definition of the Lorentz function space $\Lambda_{\varphi_E}$ my be rewritten in the form
$$\norm{f}_{\Lambda_{\varphi_E}(I)} = \norm{f}_{L^\infty(I)}\varphi_E(0^{+}) + \int_0^{m(I)}f^{*}(t)\phi_E(t)\text{d}t,$$
where $\varphi_E(t) = \int_0^t \phi_E(s)\text{d}s$.
Then we have embedding $\Lambda_{\varphi_E} \hookrightarrow E$ with embedding constant equal 1.

\begin{proposition} Let $\varphi$ be quasi-concave function. Then
$$
(C\Lambda_\varphi)_a = \left\{ \begin{array}{ll}
C\Lambda_\varphi & \textrm{if $\varphi(0^{+}) = 0, \varphi(\infty) = \infty$}\\
\{f\in C\Lambda_\varphi : \lim\limits_{t\rightarrow\infty} (Cf)^{*}(t) = 0 \} & \textrm{if $\varphi(0^{+}) = 0, \varphi(\infty) < \infty$}\\
\{f\in (C\Lambda_\varphi)_b : \lim\limits_{t\rightarrow 0^{+}} Cf(t) = 0 \} & \textrm{if $\varphi(0^{+}) > 0, \varphi(\infty) = \infty$} \\
\{f\in C\Lambda_\varphi : \lim\limits_{t\rightarrow 0^{+},\infty} Cf(t) = 0 \} & \textrm{if $\varphi(0^{+}) > 0, \varphi(\infty) < \infty$}
\end{array} \right.
$$
In particular, the Ces\`aro-Lorentz function space $C\Lambda_\varphi$ is order continuous if and only if $\varphi(0^{+}) = 0$ and $\varphi(\infty) = \infty$ (i.e. $C\Lambda_\varphi \in (OC)$ if and only if $\Lambda_\varphi \in (OC)$).
\end{proposition}
\begin{proof}
We can assume that $C\Lambda_\varphi$ is non-trivial because if $C\Lambda_\varphi = \{0\}$ there is noting to prove.
To prove the first part we have to consider the following cases. Suppose
\begin{enumerate}
\item $\varphi(0^{+}) = 0$ and $\varphi(\infty) = \infty$. It follows from \cite[Lemma 5.1]{KPS78} and
Theorem A (iv) that $\Lambda_\varphi \in (OC)$. Moreover, from \cite[Corollary 4.13]{Ko16} we have
$\frac{1}{x}\chi_{(1,\infty)}(x) \in (\Lambda_\varphi)_a$ so using Remark \ref{CXawarunekrownowazny1/x}
we obtain $(C\Lambda_\varphi)_a = C(\Lambda_\varphi)_a = C\Lambda_\varphi$.
\item $\varphi(0^{+}) = 0$ and $\varphi(\infty) < \infty$. In this situation $\Lambda_\varphi \notin (OC)$
but $(\Lambda_\varphi)_a$ is non-trivial. More precisely, 
$$(\Lambda_\varphi)_a = \{ f\in \Lambda_\varphi : \lim\limits_{t\rightarrow\infty} f^{*}(t) = 0 \},$$
see \cite[Corollary 4.13]{Ko16}. Since $C$ is bounded, using Theorem \ref{CXa} we get
$$(C\Lambda_\varphi)_a = C(\Lambda_\varphi)_a = \{f\in C\Lambda_\varphi : \lim\limits_{t\rightarrow\infty} (Cf)^{*}(t) = 0 \}.$$
\item $\varphi(0^{+}) > 0$ and $\varphi(\infty) = \infty$. From Theorem B we get $(\Lambda_\varphi)_a = \{0\}$ and
it is enough to use Theorem \ref{CXa} (ii).
\item $\varphi(0^{+}) > 0$ and $\varphi(\infty) < \infty$. In this case $\Lambda_\varphi = L^\infty$, see \cite[p. 108]{KPS78} and we can use Remark \ref{Eapuste}.
\end{enumerate}
The second part is a direct consequence of the foregoing considerations.
\end{proof}

\subsection{The Ces\`aro-Marcinkiewicz function spaces $CM_\varphi$}

For any quasi-concave function $\varphi$ on $I$ the Marcinkiewicz function space (called also weak Lorentz space) $M_\varphi = M_\varphi(I)$ is defined as
$$M_\varphi(I) := \{ f \in L^{0}(I) : \sup\limits_{t\in I} \varphi(t)f^{**}(t) < \infty \},$$
with a norm $\norm{f}_{M_\varphi(I)} = \sup_{t\in I} \varphi(t)f^{**}(t)$, where $f^{**}(t) := \frac{1}{t}\int_0^t f^{*}(s)\text{d}s$
is a maximal function. The Marcinkiewicz function space $M_\varphi$
is the largest symmetric function space with fundamental function $\varphi$, i.e. for any symmetric function space
$E$ we have $E \hookrightarrow M_{\varphi_E}$ with embedding constant equal 1.



\begin{proposition} Let $\varphi$ be quasi-concave function. Then
$$
(CM_\varphi)_a = \left\{ \begin{array}{ll}
\{f\in CM_\varphi : \lim\limits_{t\rightarrow 0^{+},\infty} \varphi(t)(Cf)^{**}(t) = 0 \} & \textrm{if $\varphi(0^{+}) = 0, p(M_\varphi)>1$}\\
(CM_\varphi)_b & \textrm{if $\varphi(0^{+}) = 0, p(M_\varphi)=1$}\\
\{f\in (CM_\varphi)_b : \lim\limits_{t\rightarrow 0^{+}} Cf(t) = 0 \} & \textrm{if $\varphi(0^{+}) > 0, \varphi(\infty) = \infty$} \\
\{f\in CM_\varphi : \lim\limits_{t\rightarrow 0^{+},\infty} Cf(t) = 0 \} & \textrm{if $\varphi(0^{+}) > 0, \varphi(\infty) < \infty$}
\end{array} \right.
$$
In particular, if we assume that $p(M_\varphi)>1$ then the non-trivial Ces\`aro-Marcinkiewicz function space $CM_\varphi$ is never order continuous.
\end{proposition}
\begin{proof}
Let us consider the following cases:
\begin{enumerate}
\item $\varphi(0^{+}) = 0$ and $p(M_\varphi)>1$. Operator $C$ is bounded from Theorem D and
$$(M_\varphi)_a = \{f \in M_\varphi : \lim\limits_{t\rightarrow 0^{+},\infty} \varphi(t)f^{**}(t) = 0 \},$$
see \cite[Def. 1.3 and Th. 1.3]{KL04} so applying Theorem \ref{CXa} we get
$$(CM_\varphi)_a = C(M_\varphi)_a = \{f \in CM_\varphi : \lim\limits_{t\rightarrow 0^{+},\infty} \varphi(t)(Cf)^{**}(t) = 0 \}.$$
\item $\varphi(0^{+}) = 0$ and $p(\Lambda_\varphi)=1$. Because $(M_\varphi)_a \neq \{0\}$ and the Ces\`aro
operator $C$ is not bounded on $M_\varphi$ we can use Theorem \ref{CXa} (i).
\item $\varphi(0^{+}) > 0$ and $\varphi(\infty) = \infty$. In this situation $(M_\varphi)_a$ is trivial in view of Theorem B,
so we simply use Theorem \ref{CXa} (ii).
\item $\varphi(0^{+}) > 0$ and $\varphi(\infty) < \infty$. It is easy to see that $M_\varphi = L^\infty$. Indeed, inclusion $M_\varphi \hookrightarrow L^\infty$ is a consequence of Theorem B. Now, if $f \in L^\infty$ then $f^{**}$ also is a bounded function and
$$\norm{f}_{M_\varphi} = \sup_{t\in I}\varphi(t)f^{**}(t) \leq \norm{f^{**}}_{L^\infty(I)}\norm{\varphi}_{L^\infty(I)} < \infty,$$ since $\varphi(\infty) < \infty$. Therefore also $L^\infty \hookrightarrow M_\varphi$. From Remark \ref{Eapuste} we get
$$(CM_\varphi)_a = (Ces_\infty)_a = \{f\in CM_\varphi : \lim\limits_{t\rightarrow 0^{+},\infty} Cf(t) = 0 \},$$
and the proof is finished.
\end{enumerate}
\end{proof}

\begin{question}
It seems that non-trivial Ces\`aro-Marcinkiewicz function space $CM_\varphi$ is never order continuous but this is
only our conjecture.
\end{question}

\subsection{The spaces $C(L^1\cap L^\infty)$ and $C(L^1 + L^\infty)$} The spaces $L^1\cap L^\infty$ and $L^1 + L^\infty$ occupy a special place
in the theory of symmetric spaces because they are respectively the largest and the smallest of all symmetric function spaces, i.e.
$L^1\cap L^\infty \hookrightarrow E[0,\infty) \hookrightarrow L^1 + L^\infty$ and $L^\infty \hookrightarrow E[0,1] \hookrightarrow L^1$, cf. \cite[Th. 6.6 and Th. 6.7, pp. 77-78]{BS88},
$\norm{f}_{L^1\cap L^\infty} = \max\{\norm{f}_{L^1},\norm{f}_{L^\infty}\}$,
 and $$\norm{f}_{L^1 + L^\infty} =\inf\{ \norm{g}_{L^1}+\norm{h}_{L^\infty} : f = g + h,\ f\in L^1\ \text{and}\ h\in L^\infty \}
= \int_0^1 f^{*}(t)\text{d}t,$$
cf. \cite[Th. 6.2 and Th. 6.4, p. 74 and 76]{BS88}.

The case of $L^1\cap L^\infty$ is not very interesting
because if $I = [0,\infty)$ then $C(L^1\cap L^\infty) = \{0\}$ and if $I = [0,1]$ then $L^\infty \hookrightarrow L^1$ so
$L^1 \cap L^\infty = L^\infty$ and $L^1 + L^\infty = L^1$.
Therefore for $I = [0,1]$ we have $C(L^1 \cap L^\infty) = C(L^\infty) = Ces_\infty$ and $C(L^1 + L^\infty) = C(L^1) = Ces_1$ - these cases were already discussed in the Section 4.1.

\begin{proposition}
$(C(L^1 + L^\infty))_a=\{f\in C(L^1 + L^\infty): (Cf)^*(\infty)=0\}$.
\end{proposition}
\begin{proof}
We claim that\begin{equation}\label{eq: 1}
(L^1 + L^\infty)_a=\{f\in L^1 + L^\infty : f^*(\infty)=0\}. 
\end{equation}
This characterization is well known (see e.g. note in Section 2 in \cite{KL04} and references therein) but we will give a short proof for the sake of completeness. Inclusion $\subset$ is a direct consequence of Theorem A (ii). For the reverse inclusion $\supset$ take $f\in L^1+L^\infty$ and
observe that
\begin{align*}
\norm{f^{*}\chi_{[0,1)}}_{L^1}  & = \int_0^1 (f^{*}\chi_{[0,1)})^{*}(t)\text{d}t\\
 &= \int_0^1 f^{*}(t)\text{d}t  = \norm{f}_{L^1 + L^\infty} < \infty.
\end{align*}
Since $f^{*}\chi_{[0,\frac{1}{n}]} \rightarrow 0$ almost everywhere and is dominated by integrable function thus, from Lebesgue dominated convergence theorem $$\norm{f^{*}\chi_{[0,\frac{1}{n}]}}_{L^1 + L^\infty}=\norm{f^{*}\chi_{[0,\frac{1}{n}]}}_{L^1} \rightarrow 0,$$ as $n \rightarrow \infty$. Therefore, if $f \notin (L^1 + L^\infty)_a$ then according to Theorem A (ii) we get  $$\norm{f^{*}\chi_{[n,\infty)}}_{L^1 + L^\infty} \nrightarrow 0.$$ Passing to subsequence if necessary we can assume that 
$\norm{f^{*}\chi_{[n,\infty)}}_{L^1 + L^\infty} \geq \delta$ for some $\delta >0$. We have
\begin{align*}
f^{*}(n) & \geq \int_n^{n+1}f^{*}(t)\text{d}t = \int_0^1 (f^{*}\chi_{[n,\infty)})^{*}(t)\text{d}t \\
                                            & = \norm{f^{*}\chi_{[n,\infty)}}_{L^1 + L^\infty} \geq \delta > 0,
\end{align*}
so $f^{*}(\infty) > 0$ and the claim follows. 
Note that from (\ref{eq: 1}) $$\frac{1}{t}\chi_{[1,\infty)}(t)\in (L^1 + L^\infty)_a.$$ Using Remark 13 we obtain $(C(L^1 + L^\infty))_a=C((L^1 + L^\infty)_a)$ and the proof is complete.
\end{proof}

\bigskip

(Tomasz Kiwerski) Faculty of Mathematics, Computer Science and Econometrics,
University of Zielona G\`{o}ra,
prof. Z. Szafrana 4a, 65-516 Zielona G\`{o}ra, Poland

\textit{e-mail address}: \verb|tomasz.kiwerski@gmail.com|

\bigskip

(Jakub Tomaszewski) Institute of Mathematics,
Faculty of Electrical Engineering,
Pozna\'{n} University of Technology,
Piotrowo 3A, 60-965 Pozna\'{n}, Poland

\textit{e-mail address}: \verb|jakub.tomaszewski42@gmail.com|


\newpage

\begin{thebibliography}{99}
\bibitem{AM09} S. V. Astashkin, L. Maligranda, \textit{Structure of Ces\`aro
function spaces}, Indag. Math. (N.S.) 20 (3) (2009) 329-379.

\bibitem{AM14} S. V. Astashkin, L. Maligranda, \textit{Structure of Ces\`aro
function spaces: a survey}, Banach Center Publ. 102 (2014) 13-40.

\bibitem{AP02} W. Arendt, B. de Pagter, \textit{Spectrum and asymptotics of the Black-Scholes partial differential equation in $(L^1, L^{\infty})$-interpolation spaces}, Pacific J. Math. 202 (1) (2002) 1-36.

\bibitem{BS88} C. Bennet, R. Sharpley, \textit{Interpolation of Operators},
Pure Appl. Math., vol. 129, Academic Press, Inc., Boston, San Diego, New
York, Berkeley, London, Sydney, Tokyo, Toronto, 1988.


\bibitem{Bo71} D. W. Boyd, \textit{Indices for the Orlicz spaces}, Pacific
J. Math. 38 (2) (1971) 315-323.

\bibitem{Bo69} D. W. Boyd, \textit{Indices of function spaces and their
relationship to interpolation}, Canad. J. Math. 21 (1969) 1245-1254.

\bibitem{Ch96} S. Chen, \textit{Geometry of Orlicz spaces}, Disserationes
Math. (Rozprawy Mat.) CCCLVI (1996).

\bibitem{CKP14} M. Ciesielski, P. Kolwicz, A. Panfil, \textit{Local monotonicity structure of symmetric spaces with application},
J. Math. Anal. Appl. 409 (2014) 649-662.


\bibitem{CR16} G. P. Curbera, W. J. Ricker, \textit{Abstract Ces\`aro spaces: integral representation}, J. Math. Anal. Appl. 441 (2016) (1) 25-44.

\bibitem{DS07} O. Delgado, J. Soria, \textit{Optimal domain for the Hardy
operator}, J. Funct. Anal. 244 (1) (2007) 119-133.

\bibitem{GH99} A. S. Granero, H. Hudzik, \textit{On some proximinal
subspaces of modular spaces}, Acta Math. Hungar. 85 (1-2) (1999) 59-79.

\bibitem{HH73} B. D. Hassard, D. A. Hussein, \textit{On Ces\`aro function spaces}, Tamkang J. Math. 4 (1973) 19-25.

\bibitem{KL04} A. Kami\'{n}ska, M. J. Lee, \textit{M-ideal properties in Marcinkiewicz spaces}, Comment. Math. (2004), 123-144,
Tomus specialist in Honorem Juliani Musielak.

\bibitem{KMP00} A. Kami\'{n}ska, M. Maligranda, L.-E. Persson, \textit{%
Indices and regularization of measurable functions}, in: Function Spaces,
The 5th Conference: Proceedings of the Conference at Pozna\'{n}, Poland
2000, 231-246

\bibitem{KK16} T. Kiwerski, P. Kolwicz, \textit{Isomorphic copies of $%
l^{\infty }$ in Ces\`aro-Orlicz function spaces}, Positivity, published online on 3 November 2016,
doi: 10.1007/s11117-016-0449-6.

\bibitem{KK17} T. Kiwerski, P. Kolwicz, \textit{Isometric copies of $%
l^{\infty }$ in Ces\`aro-Orlicz function spaces}, preprint of 16 pages, 30 January 2016,
\verb+arXiv:1701.08794v1+ at http://arxiv.org/abs/1701.08794.

\bibitem{KMP07} A. Kufner, L. Maligranda, L. E. Persson, \textit{The Hardy
inequality. About its history and some related results}, Vydavatelsky Servis
Publishing House, Pilsen 2007.

\bibitem{KR61} M. A. Krasnosel'ski\u \i, Ya. B. Ruticki\u \i, \textit{Convex
functions and Orlicz Spaces}, P. Noorddhoff Ltd., Groningen, 1961
(translation).

\bibitem{KPS78} S. G. Krein, Yu. I. Petunin, E. M. Semenov, \textit{%
Interpolation of linear operators}, Nauka, Moscow, Russia, 1978.

\bibitem{Ko05} P. Kolwicz, \textit{Rotundity properties in Calder\`on-Lozanovski\u \i}, Houston J. Math., 31 (3) (2005) 883-912.

\bibitem{Ko16} P. Kolwicz, \textit{Local structure of symmetrizations $E^{(*)}$
 with applications}, J. Math. Anal. Appl. 440 (2) (2016) doi: 10.1016/j.jmaa.2016.03.075 

\bibitem{LM15a} K. Le\'{s}nik, M. Maligranda, \textit{On abstract Ces\`aro
spaces. I. Duality}, J. Math. Anal. Appl. 424 (2) (2015) 932-951, doi:
10.1016/j.jmaa.2014.11.023.

\bibitem{LM15b} K. Le\'{s}nik, M. Maligranda, \textit{On abstract Ces\`aro
spaces. II. Optimal Range}, Integral Equations Operator Theory 81 (2) (2015)
227-235.

\bibitem{LM15p} K. Le\'{s}nik, M. Maligranda, \textit{Interpolation of
abstract Ces\`aro, Copson and Tandori spaces}, Indag. Math. (N.S.) 27 (3)
(2016), 764--785, doi:10.1016/j.indag.2016.01.009. 

\bibitem{LT77} J. Lindenstrauss, L. Tzafriri, \textit{Classical Banach
Spaces I. Sequence Spaces}, Springer-Verlag, Berlin, Heidelberg, 1977.

\bibitem{LT79} J. Lindenstrauss, L. Tzafriri, \textit{Classical Banach
Spaces II. Function Spaces}, Springer-Verlag, Berlin, Heidelberg, 1979.

\bibitem{Lo69} G. Ya. Lozanovski{\u \i}, \textit{On isomorphic Banach
structures}, Sibirsk. Math. J. 10 (1969) 93-98.

\bibitem{Ma85} L. Maligranda, \textit{Indices and interpolation},
Disserationes Math. (Rozprawy Mat.) CCXXXIV (1985).

\bibitem{Ma89} L. Maligranda, \textit{Orlicz Spaces and Interpolation}, Sem.
Mat., vol. 5, Univ. of Campinas, Campinas SP, Brazil, 1989.

\bibitem{Mu83} J. Musielak, \textit{Orlicz Spaces and Modular Spaces},
Lecture Notes in Math., vol. 1034, Springer-Verlag, 1983.


\bibitem{Sh70} J. S. Shiue, \textit{A note on Ces\`aro function space}, Tamkang J. Math. 1 (1970) 91-95.

\bibitem{Za83} A. C. Zaanen, \textit{Riesz spaces II}, North-Holland Math., Library 30, North-Holland, Amsterdam, 1983.
\end{thebibliography}
\end{document}